\documentclass{amsart}
\usepackage{amssymb,amscd,graphics,axodraw}
\usepackage[enableskew]{youngtab}
\usepackage[usenames]{color}
\usepackage[all]{xy}

\numberwithin{equation}{section}

\newtheorem{theorem}{Theorem}
\newtheorem{proposition}[theorem]{Proposition}

\theoremstyle{definition}

\newtheorem{definition}{Definition}

\numberwithin{theorem}{section}
\numberwithin{definition}{section}
\numberwithin{remark}{section}

\newcommand{\calS}{\mathcal{S}}
\newcommand{\geh}{\mathfrak{g}}

\newcommand{\la}{\lambda}
\newcommand{\La}{\Lambda}
\newcommand{\Lab}{\ol{\Lambda}}
\newcommand{\ol}{\overline}
\newcommand{\ot}{\otimes}
\newcommand{\pair}[1]{\langle{#1}\rangle}
\newcommand{\pr}{\text{pr}\;}
\newcommand{\veps}{\varepsilon}
\newcommand{\vphi}{\varphi}
\newcommand{\wt}{\mathrm{wt}}
\newcommand{\Z}{\mathbb{Z}}

\newcommand{\hdom}{{\begin{picture}(8,4)\multiput(0,0)(4,0){3}{\line(0,1){4}}
\multiput(0,0)(0,4){2}{\line(1,0){8}}\end{picture}}}
\newcommand{\vdom}{{\begin{picture}(4,8)\multiput(0,0)(0,4){3}{\line(1,0){4}}
\multiput(0,0)(4,0){2}{\line(0,1){8}}\end{picture}}}
\newcommand{\cell}{{\begin{picture}(4,4)\multiput(0,0)(4,0){2}{\line(0,1){4}}
\multiput(0,0)(0,4){2}{\line(1,0){4}}\end{picture}}}

\begin{document}

\title[Simplicity and similarity]{Simplicity and similarity of Kirillov-Reshetikhin crystals}

\author[M.~Okado]{Masato Okado}
\address{Department of Mathematical Science,
Graduate School of Engineering Science, Osaka University,
Toyonaka, Osaka 560-8531, Japan}
\email{okado@sigmath.es.osaka-u.ac.jp}

\date{December 3, 2012}

\begin{abstract}
We show that the Kirillov-Reshetikhin crystal $B^{r,s}$ for nonexceptional affine types is simple
and have the similarity property. As a corollary of the first fact we can derive that 
$B^{r_1,s_1}\ot\cdots\ot B^{r_l,s_l}$ is connected. Variations of the second property are also given.
\end{abstract}

\maketitle


\section{Introduction}
It is widely recognized that Kirillov-Reshetikhin modules \cite{KR:1990,Chari:2001,CH} have distinguished 
properties among finite-dimensional modules of quantum affine algebras. One of such properties is the 
so called $T$-systems \cite{KNS}, certain algebraic relations among their $q$-characters 
\cite{FR,Nakajima:2003,H:2007}. Another property is the conjectural existence of crystal basis 
\cite{HKOTY,HKOTT}. We call it Kirillov-Reshetikhin (KR) crystal, and denote it by $B^{r,s}$ since it is 
parametrized by $(r,s)\in I\setminus\{0\}\times\Z_{>0}$, where $I$ is the index set of the Dynkin diagram
of the affine algebra and $0$ is the index as prescribed in \cite{Kac}. In \cite{KMN2:1992} many KR crystals 
were constructed through the study of perfect crystals \cite{KMN1:1992}. 
The existence of $B^{r,1}$ for any affine type was shown in \cite{Ka:2002} and its combinatorial structure 
was clarified in \cite{NS:2006,HN}. Recently, the existence of KR crystals was settled in \cite{OS:2008} 
for all affine algebras of nonexceptional types.  The crystal structure of these KR crystals was clarified 
in \cite{FOS}. Using this result the conjecture on the perfectness 
\cite{HKOTY,HKOTT} was solved for nonexceptional types in \cite{FOS2}.

In this paper, we show two more important and naturally expected properties of KR crystals. The first one 
is the simplicity (Definition \ref{def:simple}, Theorem \ref{th:simple}) introduced in \cite{AK}. 
As a consequence, one can show
the tensor product of KR crystals $B^{r_1,s_1}\ot\cdots\ot B^{r_l,s_l}$ is connected. The second one is
the similarity (Theorem \ref{th:similarity}), which was first observed in \cite{Ka:1996} for the crystal 
$B(\la)$ of the irreducible highest weight module of highest weight $\la$. We also consider variations of 
the similarity in \S\ref{sec:var}. Combining (2-i) and (1-iii) (or (1-v)) one can settle Conjecture 6.6 in
\cite{OSS} on the alignedness of the virtual crystal.

Two remarks are in order. In September 2009 Sagaki informed the author that the simplicity can also be shown 
using the method in \cite{NS}. The connectedness of the two fold tensor product of perfect KR crystals was
shown in \cite{FSS} under certain assumptions, which were settled later in \cite{ST}.

\section{Review on KR crystals}

\subsection{Crystals}

We do not review the notion of crystal bases or crystals but refer the reader to \cite{Ka:1991,Ka:1994}.
In this paper crystal operators are denoted by $e_i,f_i$ instead of $\tilde{e}_i,\tilde{f}_i$ and 
we only consider seminormal crystals satisfying
\[
\veps_i(b) = \max\{k \in \Z_{\ge 0} \mid e_i^kb \neq 0 \},\quad
\vphi_i(b) = \max\{k \in \Z_{\ge 0} \mid f_i^kb \neq 0 \}
\]
for an element $b$ of a crystal $B$.
Throughout the paper $\geh$ stands for an affine Kac-Moody Lie algebra \cite{Kac} whose 
weight lattice, simple roots, simple coroots, fundamental weights are denoted by
$P,\{\alpha_i\}_{i\in I},\{h_i\}_{i\in I},\{\La_i\}_{i\in I}$ with index set $I$.
$\geh_0$ is the underlying finite-dimensional simple Lie algebra of $\geh$ whose index set is
$I_0:=I\setminus\{0\}$ following the convention of \cite{Kac}. $\ol{P},\{\alpha_i\}_{i\in I_0},
\{h_i\}_{i\in I_0},\{\Lab_i\}_{i\in I_0}$ are the weight lattice, simple roots, simple 
coroots, fundamental weights of $\geh_0$. $U_q(\geh),U_q(\geh_0)$ are the 
quantized enveloping algebras of $\geh,\geh_0$ and $U'_q(\geh)$ is the subalgebra of $U_q(\geh)$
without the degree operator $q^d$.

A crystal $B$ with index set $I$ is said to be {\it regular} if, for any subset $J\subset I$,
$B$ is isomorphic to the crystal bases associated to an integrable $U_q(\geh_J)$-module, where
$U_q(\geh_J)$ is the subalgebra of $U_q(\geh)$ generated by the Chevalley generators with index set $J$.
It is known \cite{Ka:1994} that the Weyl group $W$ acts on any regular crystal. Let $\{s_i\}_{i\in I}$ 
be the simple reflections of $W$. The action $\calS$ is given by
\[
\calS_{s_i}b=\left\{
\begin{array}{ll}
f_i^{\pair{h_i,\wt(b)}}b\quad&\text{if }\pair{h_i,\wt(b)}\ge0,\\
e_i^{-\pair{h_i,\wt(b)}}b\quad&\text{if }\pair{h_i,\wt(b)}\le0.
\end{array}\right.
\]

Let a crystal $B$ with index set $I$ be given and $J \subset I$. Then we can regard $B$ as a $J$-crystal by
only considering $e_i,f_i$ for $i\in J$. We also say an element $b$ is $J$-{\it highest} 
(resp. $J$-{\it lowest}) if $e_ib=0$ (resp. $f_ib=0$) for all $i\in J$.

\subsection{$X_n\to X_{n-1}$ branching and $\pm$-diagrams} \label{subsec:branching}

Let $X_n=B_n,C_n$, or $D_n$. In this subsection we review from \cite[\S3.2]{FOS} the branching rule for 
$X_n\to X_{n-1}$ involving $\pm$-diagrams introduced in \cite{Sch:2008}.
A $\pm$-diagram $P$ of shape $\La/\la$ is a sequence of partitions $\la\subset \mu \subset \La$ 
such that $\La/\mu$ and $\mu/\la$ are horizontal strips (i.e. every column contains at most one box). We
depict this $\pm$-diagram by the skew tableau of shape $\La/\la$ in
which the cells of $\mu/\la$ are filled with the symbol $+$ and
those of $\La/\mu$ are filled with the symbol $-$. Write
$\La=\mathrm{outer}(P)$ and $\la=\mathrm{inner}(P)$ for the outer and inner shapes of the
$\pm$-diagram $P$. 
When drawing partitions or tableaux, we use the French
convention where the parts are drawn in increasing order from top to bottom.

There are a couple further type-specific requirements:
\begin{enumerate}
\item
For type $C_n$ the outer shape $\La$ contains columns of height at most $n$, but the inner 
shape $\lambda$ is not allowed to be of height $n$ (hence there are no empty columns of height $n$).
\item
For type $B_n$ the outer shape $\La$ contains columns of height at most $n$; for the columns
of height $n$, the $\pm$-diagram can contain at most one $0$ between $+$ and $-$ at height $n$
and no empty columns are allowed; furthermore there may be a spin column of height $n$ and 
width $1/2$ containing $+$ or $-$.
\item
For type $D_n$ suppose $\La= k_1 \La_1 + \cdots + k_{n-1} \La_{n-1} + k_n \La_n$.
If $k_n\ge k_{n-1}$ we depict this weight by $(k_n-k_{n-1})/2$ columns of height $n$ colored 1
(where we interpret a $1/2$ column as a $\La_n$ spin column if $k_n-k_{n-1}$ is odd),
$k_{n-1}$ columns of height $n-1$, and as usual $k_i$ columns of height $i$ for $1\le i\le n-2$.
If $k_n< k_{n-1}$ we depict this weight by $(k_{n-1}-k_n)/2$ columns of height $n$ colored 2
(where we interpret a $1/2$ column as a $\La_{n-1}$ spin column if $k_{n-1}-k_n$ is odd),
$k_n$ columns of height $n-1$, and as usual $k_i$ columns of height $i$ for $1\le i\le n-2$.
We require that columns of height $n$ are colored, contain $+$, $-$, or $\mp$,
but cannot simultaneously contain $+$ and $-$; spin columns can only contain $+$ or $-$.
\end{enumerate}
Then, for an $X_n$-dominant weight $\La$, there is an isomorphism of $X_{n-1}$-crystals
\begin{align*}
  B_{X_n}(\La) \cong \bigoplus_{\substack{\text{$\pm$-diagrams $P$} \\ \mathrm{outer}(P)=\La}}
  B_{X_{n-1}}(\mathrm{inner}(P)).
\end{align*}

In fact, there is a bijection $\Phi:P\mapsto b$ from $\pm$-diagrams $P$ of shape 
$\La/\la$ to the set of $X_{n-1}$-highest elements $b$ of $X_{n-1}$-weight $\la$ as given below.  
We use Kashiwara-Nakashima (KN) tableaux \cite{KN:1994} in French convention to represent 
elements of $B_{X_n}(\La)$.
Suppose $\La$ is a dominant weight; we require that $\La$ does not contain any 
columns of height $n$ for type $D_n$. For any columns of height $n$ containing $+$, place a column
$12\ldots n$ (this includes spin columns for type $B_n$). Otherwise, place $\ol{1}$ in all positions in 
$P$ that contain a $-$, place a $0$ in the position containing $0$, and fill the remainder of all 
columns by strings of the form $23\ldots k$. We move through the columns of $b$  from top to bottom,
left to right. Each $+$ in $P$ (starting with the leftmost moving to the right ignoring $+$ at height $n$)
will alter $b$ as we move through the columns. Suppose the $+$ is at height $h$ in $P$.
If one encounters a spin column of type $B_n$, replace it by a column 
$12\ldots h \; h+2\ldots n\; \ol{h+1}$ (read from bottom to top). Otherwise, if one encounters a 
$\ol{1}$, replace $\ol{1}$ by $\ol{h+1}$. If one encounters a $2$, replace the string $23\ldots k$ by 
$12\ldots h\; h+2\ldots k$.

\subsection{KR crystals}

KR crystals $B^{r,s}$ for nonexceptional types have the following general features: As an $I_0$-crystal
$B^{r,s}$ decomposes into 
\begin{equation} \label{decomp}
B^{r,s}\simeq\bigoplus_\la B(\la),
\end{equation}
where $B(\la)$ stands for the crystal of the highest weight $U_q(\geh_0)$-module of highest weight $\la$
and the $\la$ runs over all partitions that can be obtained from the $r\times s$ (or $r\times(s/2)$ only
when $\geh=B^{(1)}_n$ and $r=n$) rectangle by removing pieces of shape $\nu$ ($(\nu,\geh_0)$ are given
in Table \ref{tab:Dynkin}). However, there are some exceptions, where $B^{r,s}$ is connected and does not
decomposes as in \eqref{decomp}. We call these nodes {\it exceptional} and they are the filled nodes in
Table \ref{tab:Dynkin}.

\begin{table}
\begin{eqnarray*}
A_{n-1}^{(1)}
&\vcenter{\xymatrix@R=1ex{
&&*{\circ}<3pt> \ar@{-}[drr]^<{\;\,0} \ar@{-}[dll] \\
*{\circ}<3pt> \ar@{-}[r]_<{1} &*{\circ}<3pt> \ar@{-}[r]_<{2} 
&{} \ar@{.}[r]&{} \ar@{-}[r] &*{\circ}<3pt> \ar@{}[r]_<{n-1} &{}}}& (\phi,A_{n-1}) \\
B_n^{(1)}
&\vcenter{\xymatrix@R=1ex{
*{\circ}<3pt> \ar@{-}[dr]^<{0} \\
& *{\circ}<3pt> \ar@{-}[r]_<{2} 
& {} \ar@{.}[r]&{}  \ar@{-}[r]_>{\,\,\,\,n-1} &
*{\circ}<3pt> \ar@{=}[r] |-{\scalebox{2}{\object@{>}}}& *{\circ}<3pt>\ar@{}_<{n} \\
*{\circ}<3pt> \ar@{-}[ur]_<{1}}}& (\vdom,B_n) \\
C_n^{(1)}
&\vcenter{\xymatrix@R=1ex{
*{\circ}<3pt> \ar@{=}[r] |-{\scalebox{2}{\object@{>}}}_<{0} 
&*{\circ}<3pt> \ar@{-}[r]_<{1} 
& {} \ar@{.}[r]&{}  \ar@{-}[r]_>{\,\,\,\,n-1} &
*{\circ}<3pt> \ar@{=}[r] |-{\scalebox{2}{\object@{<}}}
& *{\bullet}<3pt>\ar@{}_<{n}}}& (\hdom,C_n) \\
D_n^{(1)}
&\vcenter{\xymatrix@R=1ex{
*{\circ}<3pt> \ar@{-}[dr]^<{0}&&&&&*{\bullet}<3pt> \ar@{-}[dl]^<{n-1}\\
& *{\circ}<3pt> \ar@{-}[r]_<{2} 
& {} \ar@{.}[r]&{} \ar@{-}[r]_>{\,\,\,n-2} &
*{\circ}<3pt> & \\
*{\circ}<3pt> \ar@{-}[ur]_<{1}&&&&&
*{\bullet}<3pt> \ar@{-}[ul]_<{n}}}& (\vdom,D_n) \\
A_{2n}^{(2)}
&\vcenter{\xymatrix@R=1ex{
*{\circ}<3pt> \ar@{=}[r] |-{\scalebox{2}{\object@{<}}}_<{0} 
&*{\circ}<3pt> \ar@{-}[r]_<{1} 
& {} \ar@{.}[r]&{}  \ar@{-}[r]_>{\,\,\,\,n-1} &
*{\circ}<3pt> \ar@{=}[r] |-{\scalebox{2}{\object@{<}}}
& *{\circ}<3pt>\ar@{}_<{n}}}& (\cell,C_n) \\
A_{2n-1}^{(2)}
&\vcenter{\xymatrix@R=1ex{
*{\circ}<3pt> \ar@{-}[dr]^<{0} \\
& *{\circ}<3pt> \ar@{-}[r]_<{2} 
& {} \ar@{.}[r]&{}  \ar@{-}[r]_>{\,\,\,\,n-1} &
*{\circ}<3pt> \ar@{=}[r] |-{\scalebox{2}{\object@{<}}}& *{\circ}<3pt>\ar@{}_<{n} \\
*{\circ}<3pt> \ar@{-}[ur]_<{1}}}& (\vdom,C_n) \\
D_{n+1}^{(2)}
&\vcenter{\xymatrix@R=1ex{
*{\circ}<3pt> \ar@{=}[r] |-{\scalebox{2}{\object@{<}}}_<{0} 
&*{\circ}<3pt> \ar@{-}[r]_<{1} 
& {} \ar@{.}[r]&{}  \ar@{-}[r]_>{\,\,\,\,n-1} &
*{\circ}<3pt> \ar@{=}[r] |-{\scalebox{2}{\object@{>}}}
& *{\bullet}<3pt>\ar@{}_<{n}}}& (\cell,B_n)
\end{eqnarray*}
\caption{\label{tab:Dynkin}Dynkin diagrams}
\end{table}

In this subsection we review from \cite{FOS} all KR crystals for nonexceptional affine types by dividing 
into cases according to the shape $\nu$ first and then treat the exceptional nodes.

\subsubsection{Type $A_{n-1}^{(1)}$}
\label{subsubsec:A}
Since
\[
	B^{r,s} \cong B(s\Lab_r)
\]
as an $I_0$-crystal $B^{r,s}$ is identified with the set of all semistandard Young tableaux of 
rectangular shape $(s^r)$ over the alphabet $1\prec 2\prec \cdots \prec n$. 
The Dynkin diagram of $A_{n-1}^{(1)}$ has a 
cyclic automorphism $i\mapsto i+1 \pmod{n}$. The action of $e_0$ and $f_0$ is given by
\[
	e_0 = \pr^{-1} \circ e_1 \circ \pr \qquad \text{and} \qquad f_0 = \pr^{-1} \circ f_1 \circ \pr,
\]
where $\pr$ is Sch\"utzenberger's promotion operator~\cite{Schuetzenberger:1972}, which
is the cyclic Dynkin diagram automorphism on the level of crystals~\cite{Sh:2002}.
On a rectangular tableau $b \in B^{r,s}$, $\pr(b)$ is obtained from $b$ by removing all letters $n$,
adding one to each letter in the remaining tableau, using jeu-de-taquin to slide all letters up,
and finally filling the holes with $1$s.

\subsubsection{Type $B_n^{(1)}, D_n^{(1)}, A_{2n-1}^{(2)}$}
\label{subsubsec:BDA}
The cases when $r=n-1,n$ for $\geh=D_n^{(1)}$ are excluded from here since they are exceptional nodes.
As an $I_0$-crystal
\[
B^{r,s}\simeq\bigoplus_\la B(\la),
\]
where $\la$ runs over all partitions obtained from the $r\times s$ (or $r\times(s/2)$ only
when $\geh=B^{(1)}_n$ and $r=n$) rectangle by removing $\vdom$.

The Dynkin diagrams in this case all have an automorphism interchanging nodes $0$ and $1$. 
The corresponding automorphism $\sigma$ of on the level of crystals exists.
By construction $\sigma$ commutes with $e_{i}$ and $f_{i}$ for $i=2,3,\ldots,n$. 
Hence it suffices to define $\sigma$ on $\{2,3,\ldots,n\}$-highest elements. 

Because of the bijection $\Phi$ described in \S\ref{subsec:branching}, it remains to define $\sigma$ on 
$\pm$-diagrams.
For the following description of the map $\mathfrak{S}$, we further assume $r\neq n$ for $B_n^{(1)}$.
Let $P$ be a $\pm$-diagram of shape $\La/\la$. Let $c_i=c_i(\la)$ be
the number of columns of height $i$ in $\la$ for all $1\le i<r$ with
$c_0=s-\la_1$. If $i\equiv r-1 \pmod{2}$, then in $P$, above each
column of $\la$ of height $i$, there must be a $+$ or a $-$.
Interchange the number of such $+$ and $-$ symbols. If $i\equiv r
\pmod{2}$, then in $P$, above each column of $\la$ of height $i$,
either there are no signs or a $\mp$ pair. Suppose there are $p_i$
$\mp$ pairs above the columns of height $i$. Change this to
$(c_i-p_i)$ $\mp$ pairs. The result is $\mathfrak{S}(P)$, which has the
same inner shape $\la$ as $P$ but a possibly different outer shape.

The affine crystal operators $e_{0}$ and $f_{0}$ are defined as
\[
e_{0} = \sigma \circ e_{1} \circ \sigma \qquad \text{and} \qquad
f_{0} = \sigma \circ f_{1} \circ \sigma.
\]

The remaining KR crystal $B^{n,s}$ for $B_n^{(1)}$ is constructed as follows.
Let $\hat{B}^{n,s}$ be the $A_{2n-1}^{(2)}$-KR crystal. Then there exists a regular $B_n^{(1)}$-crystal
$B^{n,s}$ and a unique injective map $S:B^{n,s}\rightarrow \hat{B}^{n,s}$ such that 
\begin{alignat*}{2}
S(e_ib)&=e_i^{m_i}S(b), & S(f_ib)&=f_i^{m_i}S(b), \\
\veps_i(S(b))&=m_i\veps_i(b), & \quad \vphi_i(S(b))&=m_i\vphi_i(b)
\end{alignat*}
for $i\in I$, where $(m_i)_{0\le i\le n}=(2,2,\ldots,2,1)$.

\subsubsection{Type $C_n^{(1)}$}
\label{subsubsec:C}

The case when $r=n$ is excluded from here since it is an exceptional node.
We realize $B^{r,s}$ inside the ambient crystal 
$\hat{B}^{r,s}$ of type $A_{2n+1}^{(2)}$. Let $I=\{0,1,\ldots,n\}$ 
be the index set for type $C_n^{(1)}$ and $\hat{I}=\{0,1,\ldots,n+1\}$ be that for $A_{2n+1}^{(2)}$.  
Denote the crystal operators of $\hat{B}^{r,s}$ by $\hat{e}_i$ and $\hat{f}_i$. Then the crystal
with crystal operators defined by
\[
	e_i = \begin{cases} \hat{e}_0\hat{e}_1 & \text{if $i=0$}\\
	                                    \hat{e}_{i+1} & \text{if $1\le i\le n$}
	           \end{cases}
\quad \text{and} \quad
	 f_i = \begin{cases} \hat{f}_0\hat{f}_1 & \text{if $i=0$}\\
	                                    \hat{f}_{i+1} & \text{if $1\le i\le n$.}
	           \end{cases}
\]
turns out a regular $C_n^{(1)}$-crystal $B^{r,s}$.

\subsubsection{Type $A_{2n}^{(2)},D_{n+1}^{(2)}$} 
\label{subsubsec:A(2)D(2)}

The case when $r=n$ for is excluded from here since it is an exceptional node.
Let $\hat{B}^{r,s}$ stand for the $C_n^{(1)}$-KR crystal (The $r=n$ case is treated in 
\S\ref{subsubsec:CD(2)excep}). 
According to types $\geh=A_{2n}^{(2)},D_{n+1}^{(2)}$ define positive integers $m_i$ for $i\in I$ as
\[
(m_0,m_1,\ldots,m_{n-1},m_n) = \begin{cases}
       (1,2,\ldots,2,2) & \text{ for }A_{2n}^{(2)},\\
       (1,2,\ldots,2,1) & \text{ for }D_{n+1}^{(2)}.
\end{cases}
\]
Then there exists a regular $\geh$-crystal $B^{r,s}$ and a unique injective map
$S:B^{r,s}\longrightarrow \hat{B}^{r,2s}$ such that
\begin{alignat*}{2}
S(e_ib)&=e_i^{m_i}S(b), & S(f_ib)&=f_i^{m_i}S(b), \\
\veps_i(S(b))&=m_i\veps_i(b), & \quad \vphi_i(S(b))&=m_i\vphi_i(b)
\end{alignat*}
for $i\in I$.

\subsection{KR crystals for exceptional nodes}
\label{subsec:excep}

The KR crystal $B^{r,s}$ corresponding to an exceptional node $r$ is isomorphic to
$B(s\Lab_r)$ as an $I_0$-crystal.

\subsubsection{$B^{n,s}$ of type $C_n^{(1)},D_{n+1}^{(2)}$}
\label{subsubsec:CD(2)excep}

First consider type $C_n^{(1)}$. The elements in $B(s\Lab_n)$ are KN-tableaux of shape 
$(s^n)$. Recall from \S\ref{subsec:branching}, 
that the $J$-highest elements ($J=\{2,3,\ldots,n\}$) of shape $(s^n)$ are in bijection with $\pm$-diagrams.
Since all columns are of height $n$, each column is either filled with $+$, $-$, or $\mp$.
Hence, if there are $\ell_1$ columns containing $+$, $\ell_2$ columns containing $-$, and 
$\ell_3$ columns containing $\mp$, we may identify $\pm$-diagrams $P$ with triples 
$(\ell_1, \ell_2, \ell_3)$ such that $\ell_1 + \ell_2 + \ell_3 = s$ and $\ell_1,\ell_2,\ell_3\ge 0$. 

In order to describe the affine structure, it suffices to define $e_0,f_0$ on such triples, since 
they commute with $e_i,f_i$ for $i=2,\ldots,n$. The rule is given by
\begin{equation*}
\begin{split}
	e_0(\ell_1,\ell_2,\ell_3) &= \begin{cases}
	(\ell_1 - 1, \ell_2 + 1, \ell_3) & \text{if $\ell_1>0$,}\\
	0 & \text{otherwise,}
	\end{cases}\\
	f_0(\ell_1,\ell_2,\ell_3) &= \begin{cases}
	(\ell_1 + 1, \ell_2 - 1, \ell_3) & \text{if $\ell_2>0$,}\\
	0 & \text{otherwise.}
	\end{cases}
\end{split}
\end{equation*}

Next consider type $D_{n+1}^{(2)}$ whose classical subalgebra is of type $B_n$.
The elements in $B(s\Lab_n)$
are KN-tableaux of shape $((s/2)^n)$ when $s$ is even and of shape $(((s-1)/2)^n)$ plus an extra
spin column when $s$ is odd. By \S\ref{subsec:branching}, the $J$-highest 
elements are in bijection with $\pm$-diagrams, where columns of height $n$ can contain
$+$, $-$, $\mp$ and at most one $0$; the spin column of half width can contain $+$ or $-$.

We again encode a $\pm$-diagram $P$ as a triple $(\ell_1, \ell_2, \ell_3)$,
where $\ell_1$ is twice the number of columns containing a single $+$ sign, $\ell_2$ is twice 
the number of columns containing a single $-$ sign (where spin column are counted as $1/2$ 
columns), and $\ell_3$ is twice the number of columns containing $\mp$.
If $P$ contains a $0$-column, then $\ell_1 + \ell_2 + \ell_3 =s-2$, otherwise  
$\ell_1 + \ell_2 + \ell_3 =s$.

The action of $e_0,f_0$ on $J$-highest elements in this case is given by
\begin{equation*}
\begin{split}
	e_0(\ell_1,\ell_2,\ell_3) &= \begin{cases}
	(\ell_1, \ell_2 + 2, \ell_3) & \text{if $\ell_1 + \ell_2 + \ell_3 < s$,}\\
	(\ell_1-2, \ell_2, \ell_3) & \text{if $\ell_1 + \ell_2 + \ell_3 = s$ and $\ell_1 > 1$,}\\
	(0, \ell_2+1, \ell_3) & \text{if $\ell_1 + \ell_2 + \ell_3 = s$ and $\ell_1 = 1$,}\\
	0 & \text{if $\ell_1 + \ell_2 + \ell_3 = s$ and $\ell_1=0$,}
	\end{cases}\\
	f_0(\ell_1,\ell_2,\ell_3) &= \begin{cases}
	(\ell_1+2, \ell_2, \ell_3) & \text{if $\ell_1 + \ell_2 + \ell_3 < s$,}\\
	(\ell_1, \ell_2-2, \ell_3) & \text{if $\ell_1 + \ell_2 + \ell_3 = s$ and $\ell_2 > 1$,}\\
	(\ell_1+1, 0, \ell_3) & \text{if $\ell_1 + \ell_2 + \ell_3 = s$ and $\ell_2 = 1$,}\\
	0 & \text{if $\ell_1 + \ell_2 + \ell_3 = s$ and $\ell_2=0$.}
	\end{cases}
\end{split}
\end{equation*}

\subsubsection{$B^{n,s}$ and $B^{n-1,s}$ of type $D_n^{(1)}$}
\label{subsubsec:Dexcep}

We first introduce an involution $\sigma:B^{n,s} \leftrightarrow B^{n-1,s}$ corresponding to the 
Dynkin diagram automorphism that interchanges the nodes $0$ and $1$.
Under this involution, $J$-components ($J=\{2,3,\ldots,n\}$) need to be mapped to 
$J$-components. Hence it suffices to define $\sigma$ on 
$J$-highest elements or equivalently $\pm$-diagrams.
Recall from \S\ref{subsec:branching}, that for weights $\La=s\Lab_n$ or $s\Lab_{n-1}$,
the $\pm$-diagram can contain columns with $+$ and $\mp$ or with $-$ and $\mp$ (but not a 
mix of $-$ and $+$ columns). 
The involution $\sigma:B^{n,s}$ maps a $\pm$-diagram $P$
to a $\pm$-diagram $P'$ of opposite color where columns containing $+$ are interchanged
with columns containing $-$ and vice versa. Then the action of $e_0$ and $f_0$ is given by
\begin{equation*}
	e_0 = \sigma \circ e_1 \circ \sigma \qquad \text{and} \qquad f_0 = \sigma \circ f_1 \circ \sigma.
\end{equation*}

\section{Simplicity}

In this section we review the notion of simple crystal from \cite{AK} and show the KR crystal $B^{r,s}$ 
simple. Recall $W$ is the Weyl group. We say that an element $b$ of a regular crystal $B$ is {\it extremal}
if it satisfies the following conditions:
we can find elements $\{b_w\}_{w\in W}$ such that
\begin{align*}
&b_w=b\quad\text{ for }w=e;\\
&\text{if }\pair{h_i,w\la}\ge0,\text{ then }e_ib_w=0\text{ and }f_i^{\pair{h_i,w\la}}b_w=b_{s_iw};\\
&\text{if }\pair{h_i,w\la}\le0,\text{ then }f_ib_w=0\text{ and }e_i^{-\pair{h_i,w\la}}b_w=b_{s_iw}.
\end{align*}

\begin{definition} \label{def:simple}
We say a finite regular crystal $B$ is {\it simple} if $B$ satisfies the following:
\begin{itemize}
\item[(1)] there exists $\la\in P$ such that the weight of any extremal element of $B$ is 
contained in $W\la$;
\item[(2)] $\sharp(B_\la)=1$.
\end{itemize}
\end{definition}

Then it was shown \cite{AK} that simple cyrstals have the following properties.

\begin{proposition} 
\begin{itemize}
\item[(1)] A simple crystal is connected.
\item[(2)] The tensor product of simple crystals is also simple.
\end{itemize}
\end{proposition}

We show the first main theorem of this paper.

\begin{theorem} \label{th:simple}
The KR crystal $B^{r,s}$ of nonexceptional affine types is simple.
\end{theorem}

\begin{proof}
We can assume $\la$ in Definition \ref{def:simple} is classically dominant, namely, 
$\la\in\sum_{i\in I_0}\Z_{\ge0}\Lab_i$. Then an extremal element of weight $\la$ is necessarily
$I_0$-highest. Hence, the cases when $\geh=A_{n-1}^{(1)}$ or $r$ is an exceptional node are done.
We assume $r$ is a nonexceptional node and prove the theorem by showing any $I_0$-highest element
$b$ of weight $\la\ne s\Lab_r$ is not extremal.

Type $B_n^{(1)},D_n^{(1)},A_{2n-1}^{(2)}$: We assume $(\geh,r)\ne(B_n^{(1)},n)$ first.
Let $\la=\sum_{j\equiv r\text{ mod }2}$ $s_j\Lab_j$. When $r$ is even, set $s_0=s-\sum_js_j$.
Let $k=\min\{0\le j\le r\mid j\equiv r\text{ mod }2,s_j>0\}$. Consider the element 
$b'=\calS_2\calS_1\cdot \calS_3\calS_2\cdots \calS_{k+1}\calS_kb$. In the KN tableau representation
$b'$ differs from $b$ in that the rightmost $s_k$ columns have entries $2,3,\ldots,k+1$. By the rule 
of $\sigma$ given in \S\ref{subsubsec:BDA} $\sigma(b')$ is given as follows. The shape of $\sigma(b')$
differs from $b'$ in that the height of the rightmost $s_k$ columns is $k+2$. From the left there are
$\lfloor\frac{s_k}2\rfloor$ columns with entries $1,3,\ldots,h,\ol{2}$, there is a column with entries
$2,3,\ldots,h,\ol{2}$ if $s_k$ is odd, and in the other columns, entries are $2,3,\ldots,h,\ol{1}$, 
where $h$ is the height of the column. From this description of $\sigma(b')$ one finds that
$\veps_0(b')=\veps_1(\sigma(b'))=2s-s_k>0$ and $\vphi_0(b')=\vphi_1(\sigma(b'))=s_k>0$, thereby showing
that $b$ is not extremal.

The remaining case when $\geh=B^{(1)}_n$ and $r=n$ is clear by construction.

Type $C_n^{(1)}$: Let $\la=\sum_js_j\Lab_j$ and set $s_0=s-\sum_js_j$. 
Let $k=\min\{0\le j\le r\mid s_j>0\}$. Consider the element 
$b'=\calS_1\calS_2\cdots\calS_kb$. In the KN tableau representation
$b'$ differs from $b$ in that the rightmost $s_k$ columns have entries $2,3,\ldots,k+1$.
One calculates
\begin{equation} \label{wt0}
\pair{h_0,\wt(b)}=s_k-s.
\end{equation}

Now recall the inclusion $\iota:B^{r,s}\hookrightarrow\hat{B}^{r,s}$ in \S\ref{subsubsec:C} where
$\hat{B}^{r,s}$ is the ambient KR crystal of type $A_{2n+1}^{(2)}$. Since $\iota(b)$ is
$\{2,\ldots,n+1\}$-highest, it corresponds to a $\pm$-diagram described as follows: By inner height
we mean the height of the corresponding column of the inner shape. There are $s_h/2$ columns of inner
height $h$ with $\mp,\cdot$ each if $0\le h<r$ and $r-h$ is even, and with $+,-$ each if $r-h$ is odd.
In the KN tableau representation, viewing from left, there are columns with entries 
$1,2,\ldots,h,h+2,h+3,\ldots$ and possibly $\ol{h'+1}$ on top for some $h,h'>k$. After such columns
we encounter a column with $\ol{k+1}$ on top and/or with entries $1,2,\ldots,k,k+2,k+3,\ldots$.
Then the KN tableau representation of $\iota(b')$ differs from $\iota(b)$ by replacing the last description 
with $\ol{2}$ on top and/or with entries $2,3,\ldots$. Hence we obtain 
$\vphi_0(b')=\vphi_1(\iota(b'))=s_k/2$, and from \eqref{wt0}, $\veps_0(b')=s-s_k/2$. Since
$\veps_0(b'),\vphi_0(b')>0$, $b$ is not extremal.

Type $A_{2n}^{(2)},D_{n+1}^{(2)}$: It is clear from the previous case by construction.
\end{proof}

\section{Similarity}

Let $B(\la)$ be the crystal basis for the irreducible highest weight module with highest weight
$\la$ and let $m$ be a positive integer. In \cite{Ka:1996} Kashiwara showed the following.

\begin{theorem}$($\cite{Ka:1996}$)$ \label{th:Kashiwara}
There exists a unique injective map $S_m:B(\la)\longrightarrow B(m\la)$ satisfying 
\begin{alignat}{2}
&S_m(e_ib)=e_i^mS_m(b), & &S_m(f_ib)=f_i^mS_m(b), \label{sim1} \\
&\veps_i(S_m(b))=m\veps_i(b), & \quad &\vphi_i(S_m(b))=m\vphi_i(b) \label{sim2}
\end{alignat}
for $i\in I$ and $b\in B(\la)$. Here $S_m(0)$ is understood to be $0$.
\end{theorem}
Note that \eqref{sim2} implies $\wt(S_m(b))=m\wt(b)$. Our second main theorem states that similar
properties hold also for KR crystals.

\begin{theorem} \label{th:similarity}
There exists a unique injective map
\[
S_m:B^{r,s}\longrightarrow B^{r,ms}
\]
satisfying \eqref{sim1} and \eqref{sim2} for $i\in I$ and $b\in B^{r,s}$.
\end{theorem}

\begin{proof}
Thanks to Theorem \ref{th:Kashiwara} the map $S_m$ is uniquely determined by \eqref{sim1} and 
\eqref{sim2} for $i\in I_0$, since $B^{r,s}$ is multiplicity free as $I_0$-crystal and $I_0$-highest
elements should be mapped to $I_0$-highest ones again. Hence it remains to show that the map $S_m$ so 
determined satisfies \eqref{sim1} and \eqref{sim2} for $i=0$. We check it case by case, treating the 
exceptional node case separately.

Type $A_n^{(1)}$: In the semistandard tableau representation the map $S_m$ is given by replacing 
each node having entry $a$ with $m$ nodes having entry $a$ concatinated horizontally. 
By the explicit combinatorial procedure of jeu de taquin, one can show $\pr$ commutes with $S_m$.
Then we have
\begin{align*}
S_m(e_0b)=&S_m((\pr^{-1}\circ e_1\circ\pr)(b))=(\pr^{-1}\circ e_1^m\circ\pr)S_m(b)=e_0^mS_m(b),\\
\veps_0(S_m(b))&=\veps_1(\pr(S_m(b)))=\veps_1(S_m(\pr(b)))=m\veps_1(\pr(b))=m\veps_0(b).
\end{align*}
The other relations are shown similarly.

Type $B_n^{(1)},D_n^{(1)},A_{2n-1}^{(2)}$: Similarly to $A_{n-1}^{(1)}$ case, it is enough to show
that $S_m$ commutes with the involution $\sigma$. Since $\sigma$ commutes with $e_i$ and $f_i$ 
($i=2,\ldots,n$), it is reduced to showing that $S_m$ commutes with $\sigma$ for any 
$\{2,\ldots,n\}$-highest element $b$. Let $b$ correspond to a $\pm$-diagram $P$. Then, by Proposition
2.2 of \cite{OSa} $S_m(b)$ corresponds to the $\pm$-diagram $P'$, where the number of $\mp,+,-,\cdot$ on
the columns of the inner shape of the same height is multiplied by $m$. Hence $S_m$ commutes with $\sigma$.
Note that it is valid also for $B^{n,s}$ of $B_n^{(n)}$.

Type $C_n^{(1)}$: We consider the inclusion $\iota:B^{r,s}\hookrightarrow\hat{B}^{r,s}$ in 
\S\ref{subsubsec:C} where $\hat{B}^{r,s}$ is the ambient KR crystal of type $A_{2n+1}^{(2)}$. Then we 
have 
\[
\xymatrix{B^{r,s}\ar[r]^\iota & \hat{B}^{r,s}\ar[r]^{S_m} & \hat{B}^{r,ms}}.
\]
Since $\sigma$ commutes with $S_m$ on $\hat{B}^{r,s}$, the image of the above composition is invariant
under $\sigma$. Hence it belongs to $B^{r,ms}$, thereby defining $S_m$ for $B^{r,s}$. Then we have 
\begin{align*}
S_m(e_0b)&=S_m(\hat{e}_0\hat{e}_1b)=\hat{e}_0^m\hat{e}_1^mS_m(b)=e_0^m(b),\\
\veps_0(S_m(b))&=\hat{\veps}_0(S_m(b))=m\hat{\veps}_0(b)=m\veps_0(b).
\end{align*}
Calculations for $f_0$ and $\vphi_0$ are similar.

Type $A_{2n}^{(2)},D_{n+1}^{(2)}$: Since $e_0$ ($f_0$) commutes with $e_i$ ($f_i$) ($i=2,\ldots,n$),
and similar relations for $f_0,\vphi_0$, for any $\{2,\ldots,n\}$-highest element $b$. Recall the 
construction in \S\ref{subsubsec:A(2)D(2)} and consider the following diagram
\[
\xymatrix{
B^{r,s}\ar[r]^S\ar@{.>}[d] & \hat{B}^{r,2s}\ar[d]^{\hat{S}_m} \\
B^{r,ms}\ar[r]^{S'} & \hat{B}^{r,2ms}}
\]
where $\hat{B}^{r,s}$ is the ambient $C_n^{(1)}$-KR crystal, $S,S'$ are the injective maps in 
\S\ref{subsubsec:A(2)D(2)} and $\hat{S}_m$ is the map just constructed for type $C_n^{(1)}$. For 
$b\in B^{r,s}$ the $\pm$-diagrams corresponding to $S(b)$ and $S'(b)$ both have even number of $\mp,+,-$
or $\cdot$ on the columns of the inner shape of the same height. Hence it is clear that there exists a 
map $S_m$ (broken line in the diagram) that makes the diagram commutative. Therefore the assertion 
follows from the properties of $\hat{S}_m$.

Exceptional nodes: Similarly to the previous case, it is enough to show the desired properties for any 
$\{2,\ldots,n\}$-highest element. However, it is clear from the formulas given in \S\ref{subsec:excep}.
\end{proof}

\section{Variations} \label{sec:var}

We give variations of Theorem \ref{th:similarity}. Since we treat KR crystals of different affine types,
we signify the type $\geh$ as $B^{r,s}_{\geh}$.

\begin{theorem}
(1) For each case below there is a unique injective map $S$ satisfying 
\begin{alignat*}{2}
&S(e_ib)=e_i^{m_i}S(b), & &S(f_ib)=f_i^{m_i}S(b), \\
&\veps_i(S(b))=m_i\veps_i(b), & \quad &\vphi_i(S(b))=m_i\vphi_i(b)
\end{alignat*}
for $i\in I$. We set $c_r=1\,(r\ne n),=2\,(r=n)$.

\begin{itemize}
\item[(i)] $S:B^{r,s}_{B_n^{(1)}}\longrightarrow B^{r,2s/c_r}_{A_{2n-1}^{(2)}}$ with
	$(m_i)_{i\in I}=(2,\ldots,2,1)$
\item[(ii)] $S:B^{r,s}_{C_n^{(1)}}\longrightarrow B^{r,s}_{A_{2n}^{(2)}}$ with
	$(m_i)_{i\in I}=(2,1,\ldots,1)$
\item[(iii)] $S:B^{r,s}_{C_n^{(1)}}\longrightarrow B^{r,c_rs}_{D_{n+1}^{(2)}}$ with
	$(m_i)_{i\in I}=(2,1,\ldots,1,2)$
\item[(iv)] $S:B^{r,s}_{A_{2n}^{(2)}}\longrightarrow B^{r,2s}_{C_n^{(1)}}\;(r\ne n)$ with
	$(m_i)_{i\in I}=(1,2,\ldots,2)$
\item[(v)] $S:B^{r,s}_{A_{2n}^{(2)}}\longrightarrow B^{r,s}_{D_{n+1}^{(2)}}\;(r\ne n)$ with
	$(m_i)_{i\in I}=(1,\ldots,1,2)$
\item[(vi)] $S:B^{r,s}_{A_{2n-1}^{(2)}}\longrightarrow B^{r,s}_{B_n^{(2)}}$ with
	$(m_i)_{i\in I}=(1,\ldots,1,2)$
\item[(vii)] $S:B^{r,s}_{D_{n+1}^{(2)}}\longrightarrow B^{r,2s/c_r}_{C_n^{(1)}}$ with
	$(m_i)_{i\in I}=(1,2,\ldots,2,1)$
\item[(viii)] $S:B^{r,s}_{D_{n+1}^{(2)}}\longrightarrow B^{r,2s/c_r}_{A_{2n}^{(2)}}$ with
	$(m_i)_{i\in I}=(2,\ldots,2,1)$
\end{itemize}

(2) Let $I,\hat{I}$ be the index set of the Dynkin diagram of $\geh,\hat{\geh}$. Let $\xi$ be a map from
$\hat{I}$ to $I$. For each case below then there is a unique injective map $S$ satisfying 
\begin{alignat*}{2}
&S(e_ib)=(\prod_{j\in\xi^{-1}(i)}\hat{e}_j)S(b),\quad& &S(f_ib)=(\prod_{j\in\xi^{-1}(i)}\hat{f}_j)S(b),\\
&\veps_i(b)=\hat{\veps}_j(S(b)),& &\vphi_i(b)=\hat{\vphi}_j(S(b))\quad\text{for any $j\in\xi^{-1}(i)$}
\end{alignat*}
for $i\in I$.

\begin{itemize}
\item[(i)] $S:B^{r,s}_{C_n^{(1)}}\longrightarrow B^{r,s}_{A_{2n+1}^{(2)}}\;(r\ne n),\;
	\xi(j)=\left\{\begin{array}{ll}
	0&(j=0,1)\\
	j-1&(2\le j\le n+1)
	\end{array}\right.$
\item[(ii)] $S:B^{r,s}_{A_{2n-1}^{(2)}}\longrightarrow 
	\left\{\begin{array}{ll}
	B^{r,s}_{D_{n+1}^{(1)}}&(r\ne n)\\
	B^{n,s}_{D_{n+1}^{(1)}}\ot B^{n+1,s}_{D_{n+1}^{(1)}}&(r=n)
	\end{array}\right.,\;
	\xi(j)=\left\{\begin{array}{ll}
	j&(0\le j\le n)\\
	n&(j=n+1)
	\end{array}\right.$
\item[(iii)] $S:B^{r,s}_{D_{n+1}^{(2)}}\longrightarrow 
	\left\{\begin{array}{ll}
	B^{r,s}_{A_{2n-1}^{(1)}}\ot B^{2n-r,s}_{A_{2n-1}^{(1)}}&(r\ne n)\\
	B^{n,s}_{A_{2n-1}^{(1)}}&(r=n)
	\end{array}\right.,\;
	\xi(j)=\left\{\begin{array}{ll}
	j&(0\le j\le n)\\
	2n-j&(n<j<2n)
	\end{array}\right.$
\end{itemize}
\end{theorem}

For instance, (1-ii) can be shown by considering them as $\{0,1,\ldots,n-1\}$-crystals and applying
\cite[Theorem 5.1]{Ka:1996}. Proof of (2-iii) is given in \cite{OSS,NS}. Other cases are shown similarly
and left to the reader. We can also consider an inclusion from $A_{2n}^{(2)\dagger}$ or 
$A_{2n-1}^{(2)\dagger}$-KR crystals, whose labeling of Dynkin nodes is opposite from $A_{2n}^{(2)}$ or
$A_{2n-1}^{(2)}$, but we do not list them here.

\subsection*{Acknowledgements}
The author thanks Ghislain Fourier and Anne Schilling for an exciting collaboration in \cite{FOS}. 
He is partially supported by the Grants-in-Aid 
for Scientific Research No. 23340007 and No. 23654007 from JSPS.

\end{document}